\documentclass{scrartcl}

\usepackage[english]{babel}
\usepackage[colorlinks]{hyperref}
\usepackage{amsfonts}
\usepackage{amsmath}
\usepackage{amsthm}
\usepackage{amssymb}
\usepackage{float}
\usepackage{rotating}
\usepackage[autostyle]{csquotes}
\usepackage{url}
\usepackage{microtype}
\usepackage{floatpag}

\DeclareMathOperator{\PGammaL}{P\Gamma{}L}

\DeclareMathOperator{\rk}{rk}
\DeclareMathOperator{\Aut}{Aut}

\newcommand{\gaussmnum}[3]{\left[\begin{smallmatrix}{#1}\\{#2}\end{smallmatrix}\right]_{#3}}
\newcommand{\gaussmset}[2]{\left[\begin{smallmatrix}{#1}\\{#2}\end{smallmatrix}\right]}
\newcommand{\F}{\ensuremath{\mathbb{F}}}
\newcommand{\Fq}{\ensuremath{\F_q}}
\newcommand{\Fqn}{\ensuremath{V}}
\newcommand{\PFqn}{\ensuremath{\operatorname{L}(\Fqn)}}
\newcommand{\Gqnk}{\ensuremath{\gaussmset{\Fqn}{k}}}
\newcommand{\I}[2]{\ensuremath{\mathcal{I}\left(#1,#2\right)}}
\newcommand{\dham}{\mathrm{d}_{\mathrm{Ham}}}
\newcommand{\rdist}{\mathrm{d}_{\mathrm{r}}}
\newcommand{\sdist}{\mathrm{d}_{\mathrm{s}}}

\newtheorem{theorem}{Theorem}
\newtheorem{lemma}{Lemma}
\newtheorem{corollary}{Corollary}

\newtheorem{fact}{Fact}

\title{Classifying optimal binary subspace codes of length 8, constant dimension 4 and minimum distance 6}

\author{Daniel Heinlein \and Thomas Honold \and Michael Kiermaier \and Sascha Kurz \and Alfred Wassermann
\thanks{
Supported by the grants KU 2430/3-1, WA 1666/9-1 -- ``Integer Linear Programming Models for Subspace Codes and 
Finite Geometry'' -- from the German Research Foundation and by Grant
no.~61571006 -- ``Research on Subspace Codes and
Related Combinatorial Structures'' -- from the National Natural Science Foundation of China.
\newline
D. Heinlein, M. Kiermaier, S. Kurz, A. Wassermann: Department of Mathematics, University of Bayreuth, Bayreuth, Germany; firstname.lastname@uni-bayreuth.de
\newline
T. Honold: ZJU-UIUC Institute, Zhejiang University, Haining, China; \href{mailto:honold@zju.edu.cn}{honold@zju.edu.cn}
\newline
The final publication is available at \url{https://link.springer.com/article/10.1007\%2Fs10623-018-0544-8}.
}
}

\begin{document}
\maketitle

\begin{abstract}
\noindent
We determine the maximum size $A_2(8,6;4)$ of a binary subspace code of packet length $v=8$, minimum subspace distance
$d=6$, and constant dimension $k=4$ to be $257$.
There are two isomorphism types of optimal codes.
Both of them are extended LMRD codes.
In finite geometry terms, the maximum number of solids in 
$\operatorname{PG}(7,2)$ mutually intersecting in at most a point is $257$.
The result was obtained by combining the classification of
substructures
with integer linear programming techniques.
This result implies that the maximum
  size $A_2(8,6)$ of a binary mixed-dimension subspace code of packet length $8$ and minimum subspace distance~$6$ is $257$ as well.


\end{abstract}

\section{Introduction}

Let $q>1$ be a prime power, {\Fq} the field with $q$ elements, and $V\cong{\F}_q^v$ a $v$-dimensional vector space over {\Fq}. 
By ${\PFqn}$ we denote the set of all subspaces of $V$.
The set ${\PFqn}$ forms a lattice with respect to the inclusion order
$U \le W \Leftrightarrow U \subseteq W$, the lattice of flats of the
projective geometry  
$\operatorname{PG}(V)\cong
\operatorname{PG}(\Fq^v)=\operatorname{PG}(v-1,q)$, and
a metric space with respect to the subspace distance
\[
\sdist(U,W)= \dim(U+W)-\dim(U \cap W)=\dim(U)+\dim(W)-2\dim(U \cap W),
\]
which
may be viewed as a $q$-analogue of the Hamming space $(\F_2^v,\dham)$.

The metric space $(\PFqn,\sdist)$ plays an important role in network coding.
It was introduced as part of the subspace channel model
in~\cite{MR2451015} to describe error-resilient data transmission in packet networks employing random linear network coding.

For $k\in\{0,\ldots,v\}$, {\Gqnk} denotes the set of all $k$-dimensional subspaces in {\Fqn}. 
We have
\[
\#\Gqnk=\gaussmnum{v}{k}{q}= \prod_{i=1}^{k} \frac{q^{v-k+i}-1}{q^i-1}\text{.}
\]

A subset $\mathcal{C}$ of $\gaussmset{V}{k}$ is called a $k$-dimensional \emph{constant-dimension code} (\emph{CDC}).
As usual, the elements of $\mathcal{C}$ are called \emph{codewords}.
For $\#\mathcal{C} \geq 2$, the minimum distance of $\mathcal{C}$ is defined as $\sdist(\mathcal{C}) = \min\{\sdist(U,W) \mid U,W\in\mathcal{C}, U\neq W\}$.
The most important parameters of a CDC $\mathcal{C}$ are
the order $q$ of the base field, the dimension $v$ of
  the ambient space $V$, the minimum (subspace) distance $d =
  \sdist(\mathcal{C})$ of $\mathcal{C}$, the cardinality
  $N=\#\mathcal{C}$, and the constant dimension $k$ of each element in
  $\mathcal{C}$. We denote them by $(v,N,d;k)_q$.
In a $(v,N,d;k)_q$ CDC the minimum distance $d$ is always an even number satisfying $2\le d\le 2\min\{k,v-k\}$.

The determination of the corresponding maximal size
$A_q(v,d;k)$ and the classification of the optimal codes is known as the \emph{main problem of subspace coding}, since it forms a $q$-analogue of the \emph{main problem of classical coding theory} (cf.~\cite[page~23]{MR0465509}).

Without the restriction of all codewords having the same
  dimension, i.e., $\mathcal{C}\subseteq\PFqn$, the code
  $\mathcal{C}$ is called a \emph{subspace code} (per se) or a \emph{mixed-dimension code} (\emph{MDC}).
The maximal cardinality of an MDC in $V$ having subspace distance $d$ is denoted as $A_q(v,d)$.
Clearly, $A_q(v,d;k) \leq A_q(v,d)$ for all $k$.

In the following, let $\beta$ be a fixed non-degenerate symmetric bilinear form on $V$ and $\pi : L(V) \to L(V), U \mapsto U^\perp$ the corresponding polarity.
The \emph{orthogonal code} or \emph{dual code} of a subspace code $\mathcal{C}$ is defined as
\[
	\mathcal{C}^\perp = \pi(\mathcal{C}) = \{U^\perp \mid U \in \mathcal{C}\}\text{.}
\]
Up to isomorphism of subspace codes as defined further below, the code $\mathcal{C}^\perp$ does not depend on the particular choice of $\beta$.

Considering orthogonal codes allows us to almost halve the parameter space:
If $\mathcal{C}$ is a $(v,N,d;k)_q$ CDC then $\mathcal{C}^\perp$ has the parameters $(v,N,d;v-k)_q$, i.e., $A_q(v,d;k)=A_q(v,d;v-k)$, so that we can assume $k\le \frac{v}{2}$ in the following.
The iterative application of the so-called 
Johnson type bound II (\cite[Theorem~3]{xia2009johnson},~\cite[Theorem~4,5]{MR2810308}), which is a $q$-generalization of 
\cite[Inequality~(5)]{johnson1962new}, gives the upper bound 
\begin{equation}
\label{ie_r_johnson}
A_q(v,d;k) \le
\left\lfloor \frac{q^{v}-1}{q^{k}-1} \left\lfloor \frac{q^{v-1}-1}{q^{k-1}-1} \left\lfloor \ldots
\left\lfloor \frac{q^{v'+1}-1}{q^{\frac{d}{2}+1}-1} A_q(v',d;\frac{d}{2}) \right\rfloor
\ldots \right\rfloor \right\rfloor \right\rfloor
\end{equation}
where $v' = v - k +
\frac{d}{2}$.  It is attained with equality at
$v=ak$ and
  $d=2k$, i.e., for spreads, and also at $v=13$, $k=3$,
  $d=4$ with
  $A_2(13,4;3)=1597245$, see~\cite{MR3542513}.  Using
  $q^r$-divisible linear codes over
  $\F_q$ with respect to the Hamming metric, this bound was sharpened
  very recently, see~\cite{kiermaier2017improvement}, to
\begin{equation}
\label{ie_best_upper_bound}
A_q(v,d;k) \!\le
\!\left\{\! \frac{q^{v}-1}{q^{k}-1} \!\left\{\! \frac{q^{v-1}-1}{q^{k-1}-1} \!\left\{\! \!\ldots\!
\!\left\{\! \frac{q^{v'+1}-1}{q^{\frac{d}{2}+1}-1} A_q(v',d;\frac{d}{2}) \!\right\}_{\!\!\frac{d}{2}+1}
\!\!\!\!\!\!\!\ldots \!\right\}_{\!\!k-2} \!\right\}_{\!\!k-1}\! \right\}_{\!\!k}\!\!,
\end{equation}
where $\left\{a / \gaussmnum{k}{1}{q}\right\}_k:=b$ with maximal $b\in\mathbb{N}$ permitting 
a representation of $a-b\cdot \gaussmnum{k}{1}{q}$ as non-negative integer combination of the summands $q^{k-1-i}\cdot\frac{q^{i+1}-1}{q-1}$ 
for $0\le i\le k-1$.\footnote{As an example we consider $A_2(9,6;4)\le 
\left\{\gaussmnum{9}{1}{2}A_2(8,6;3)/\gaussmnum{4}{1}{2}\right\}_4=
\left\{\frac{17374}{15}\right\}_4$, using $A_2(8,6;3)=34$. We have $\left\lfloor \frac{17374}{15} \right\rfloor=1158$,  
$17374-1158\cdot 15=4$, $17374-1157\cdot 15=19$, and $17374-1156\cdot 15=34$. Since $4$ and $19$ cannot be written as 
a non-negative linear combination of $8$, $12$, $14$, and $15$, but $34=14+12+8$, we have $A_2(9;6;4)\le 1156$, which improves 
upon the iterative Johnson bound by two. Let us remark that~\cite{kiermaier2017improvement} contains an easy and fast algorithm to check 
the representability as non-negative integer linear combination as specified above.} Of course, Inequality~(\ref{ie_r_johnson}) is implied by 
Inequality~(\ref{ie_best_upper_bound}). Both bounds refer back to bounds for so-called partial spreads, 
i.e., $(v,N,2k;k)_q$ codes, for which the minimum distance has the maximal value $d=2k$. For upper bounds in this special subclass of CDCs, there is a 
recent series of improvements~\cite{kurz2017improved,kurz2017packing,nastase2016maximum}. The 
underlying techniques can be explained using the language of projective $q^{k-1}$-divisible codes and the linear programming 
method, see~\cite{honold2016partial}. While a lot of upper bounds for the maximum sizes of CDCs have been proposed in the 
literature, most of them are dominated by Inequality~(\ref{ie_r_johnson}), see~\cite{heinlein2017asymptotic}. Indeed, besides 
Inequality~(\ref{ie_best_upper_bound}), the only known improvements were $A_2(6,4;3)=77<81$~\cite{MR3329980} and 
$A_2(8,6;4)\le 272<289$~\cite{new_bounds_subspaces_codes}. The latter result is improved
in this paper. For numerical values of the known lower and upper bounds on
the sizes of subspace codes we refer the reader to the online tables at \url{http://subspacecodes.uni-bayreuth.de} 
associated with~\cite{HKKW2016Tables}. The tables in particular contain representatives for the two isomorphism types of $(8,257,6;4)_2$ CDCs. A survey on Galois geometries and coding theory can be found in~\cite{Etzion2016}. 

This article investigates binary CDCs with $v=8$, $d=6$ and $k=4$.
The so-called Echelon--Ferrers construction, see e.g.~\cite{MR2589964}, gives $A_2(8,6;4)\ge 257$.
More precisely, a corresponding code is given by a lifted maximum rank distance code (LMRD code), extended by a single codeword.
In Corollary~\ref{cor_isomorphism_types} we will show that up to
isomorphism there are two such codes.
By~\cite[Theorem~10]{MR3015712}, this construction is optimal for subspace codes containing an LMRD code.
Our main theorem states that this construction is optimal even without
the restriction of containing an LMRD code and,
moreover, that all subspace codes of maximum possible size $257$ are
extended LMRD codes.

\begin{theorem}
\label{main_thm}
 $A_2(8,6;4) = 257$, and up to isomorphism there are two maximum codes, both are extended LMRD codes.
\end{theorem}

Theorem~\ref{main_thm} is the main theorem of this paper.

\begin{fact}[{\cite[Theorem~3.3(i)]{MR3543542}}]
\label{fact}
If $v=2k \ge 8$ then $A_q(v,v-2)=A_q(v,v-2;k)$.
\end{fact}

Theorem~\ref{main_thm} and Fact~\ref{fact}
  together give the maximum cardinality
  in the corresponding mixed-dimension case:

\begin{corollary}
$A_2(8,6)=257$.
\end{corollary}

Given Theorem~\ref{main_thm},
one may ask whether there exists an integer $k\ge 4$ with $A_2(2k,2k-2;k)>2^{2k}+1$.

The remaining part of the paper is structured as follows. In Section~\ref{sec_preliminaries} we provide the necessary preliminaries on 
lifted maximum rank distance codes, acting symmetry groups, and upper bounds for code sizes based on the number of 
incidences of codewords with a fixed subspace. As in~\cite{MR3329980}, we want to apply integer linear programming methods 
in order to determine the exact maximum size of CDCs with the specified parameters. Since this algorithmic approach 
suffers from the presence of a large symmetry group\footnote{Algorithmic methods taking into account known symmetries of integer linear 
programming formulations automatically are presented in the literature. However, we are not aware of any paper, where those approaches 
have been successfully applied to compute tightened upper bounds for CDCs.}, we use the inherent symmetry to prescribe 
some carefully chosen substructures up to isomorphism. A general outline of the proof of Theorem~\ref{main_thm} is presented in Section~\ref{sec_general}.
The substructures involved are described in Section~\ref{sec_substructures} and the 
integer linear programming formulations are described in Section~\ref{sec_ILP}. All these parts are put together in the proof of our 
main theorem in Section~\ref{sec_main_thm}.  

\section{Preliminaries}
\label{sec_preliminaries}
Let $m,n$ be positive integers.  The \emph{rank distance} of
$m\times n$ matrices $A$ and $B$ over $\F_q$ is defined as
$\rdist(A,B)=\operatorname{\rk}(A-B)$. The rank distance provides a
metric on $\F_q^{m\times n}$.  Any subset $C$ of the metric space
$(\F_q^{m\times n},\rdist)$ is called \emph{rank metric code}.  Its
minimum distance $d$ is the minimum of the rank distances between pairs
of distinct codewords (defined for $\#C \ge 2$). If $C$ is a subspace
of the $\F_q$-vector space $\F_q^{m\times n}$, then $C$ is called
\emph{linear}. If $m\le n$ (otherwise transpose), then
$\# C\le q^{(m-d+1)n}$ by~\cite[Theorem
5.4]{delsarte1978bilinear}. Codes achieving this bound are called
\emph{maximum rank distance} (MRD) codes.  In fact, MRD codes do
always exist. A suitable construction has independently been found
in~\cite{delsarte1978bilinear,gabidulin1985theory,roth1991maximum}. Today
these codes are known as \emph{Gabidulin codes}.  In the square case
$m=n$, after the choice of an $\F_q$-basis of
$\F_{q^n}$ the Gabidulin code is given by the matrices representing
the $\F_q$-linear maps given by the $q$-polynomials
$a_0 x^{q^0}+ a_1 x^{q^1} +\dots+ a_{n-d} x^{q^{n-d}} \in
\F_{q^n}[x]$. The lifting map
$\Lambda\colon \F_q^{m\times n}\to \gaussmset{\F_q^{m+n}}{m}$ maps an
$(m\times n)$-matrix $A$ to the row space $\langle (I_m | A)\rangle$,
where $I_m$ denotes the $m\times m$ identity matrix.  The mapping
$\Lambda$ is injective and its image is given by all $m$-dimensional
subspaces of $\F_q^{m\times n}$ having trivial intersection with the
special subspace $S=\langle e_{m+1},\dots,e_{m+n}\rangle$ of
$\F_q^{m+n}$ ($e_i$ denoting the $i$th unit vector).  In fact, the
lifting map defines an isometry from
$(\F_q^{m\times n},2\rdist)$ into
$(\operatorname{L}(\F_q^{m+n}),\sdist)$. Of particular interest are
the LMRD codes, i.e., CDCs obtained by lifting MRD codes, which are
CDCs of fairly large, though not of maximal size.

Although we use the algebraic dimension $v$ instead of the geometric
dimension $v-1$ in this paper, we adopt the use of
  geometric language: Abbreviating $k$-dimensional subspaces as
$k$-spaces, we call $1$-spaces points, $2$-spaces lines, $3$-spaces
planes, $4$-spaces solids, and $(v-1)$-spaces hyperplanes.

For dimensions $v\ge 3$ the automorphism group of the metric space
$({\PFqn},\sdist)$ is generated by $\PGammaL(\Fqn)$ and the
polarity $\pi$.  It carries the structure of a
semidirect product
$\PGammaL(\Fqn)\rtimes \langle \pi\rangle \cong \PGammaL(v,q)\rtimes
\mathbb{Z}/2\mathbb{Z}$.  Hence, for classifications of CDCs in
$\gaussmset{V}{k}$ up to isomorphism, the relevant acting group is
$\PGammaL(\Fqn)$, except for the case $v = 2k$ in which it is the
larger group $\PGammaL(\Fqn)\rtimes \langle \pi\rangle$.

In order to describe suitable substructures of $(8,N,6;4)_2$ codes with $\textit{large}$ cardinality $N$, we will consider incidences with 
fixed subspaces. To this end, let $\I{\mathcal{S}}{X}$ be the set of subspaces in $\mathcal{S} \subseteq \PFqn$ that are incident with $X \le \Fqn$, i.e.,
$\I{\mathcal{S}}{X} = \{U \in \mathcal{S} \mid U \le X \,\lor\, X \le U\}$. As special subspaces we explicitly label a point 
$\widetilde{P}=\langle (0,0,0,0,0,0,0,1) \rangle$ and a hyperplane $\widetilde{H}=\{x \in V \mid x_8 = 0\}$.
Note that $\widetilde{P}$ and $\widetilde{H}$ are not incident. By $\iota: \mathbb{F}_2^7 \rightarrow \widetilde{H}$ we denote the canonical 
embedding, which we will apply to subspaces and sets of subspaces.
 
To keep the paper self-contained, we
restate upper bounds for $\# \I{\mathcal{S}}{X}$ and $N$ from the earlier conference 
paper~\cite{new_bounds_subspaces_codes} with their complete but short proofs.

\begin{lemma}\label{lem:deg_upper_bound}
Let $\mathcal{C}$ be a $(v,\#\mathcal{C},d;k)_q$ CDC and $X \le \Fqn$. Then we have
$\#\I{\mathcal{C}}{X}\le A_q(\dim(X),d;k)$ if $\dim(X)\ge k$ and $\#\I{\mathcal{C}}{X}\le A_q(v-\dim(X),d;k-\dim(X))$ otherwise.   
\end{lemma}
\begin{proof}
Note that $\I{\mathcal{C}}{X}$ is a $(\dim(X),\#\I{\mathcal{C}}{X},d;k)_q$ CDC.
For the second part we write $V=X\oplus V'$ and $U_i=X\oplus U_i'$ for all $U_i\in \I{\mathcal{C}}{X}$. With this we have 
$\sdist(U_i,U_j)=2k-2\dim(U_i\cap U_j)\le 2\left(k-\dim(X)\right)-2\dim(U_i'\cap U_j')=\sdist(U_i',U_j')$.

\end{proof}

\begin{corollary}
\label{cor_one_incidence}
Let $\mathcal{C}$ be a $(2k,\#\mathcal{C},2k-2;k)_q$ CDC for $k\ge 1$.
Then $\#\I{\mathcal{C}}{H} \le q^k+1$ and $\#\I{\mathcal{C}}{P} \le q^k+1$ for all hyperplanes $H$ and points $P$.
\end{corollary}
\begin{proof}
We have $A_q(v,2k;k) = \frac{q^v-q}{q^k-1}-q+1$ for $v \equiv 1 \pmod{k}$ and $2 \le k \le v$, see~\cite{MR0404010}, so that 
Lemma~\ref{lem:deg_upper_bound} gives $\#\I{\mathcal{C}}{P} \le A_q(2k-1,2k-2;k-1)=q^k+1$ 
and 
$\#\I{\mathcal{C}}{H} \le A_q(2k-1,2k-2;k)=A_q(2k-1,2k-2;k-1)=q^k+1$. 

\end{proof}

In particular, Corollary~\ref{cor_one_incidence} shows that each point and hyperplane is incident with at most $17$ codewords of an $(8,N,6;4)_2$ CDC.
The next lemma refines this counting by including points which are not incident with a fixed hyperplane.

\begin{lemma}
\label{lem_minilemma}
Let $\mathcal{C}$ be an $(8,\#\mathcal{C},6;4)_2$ CDC of size $\#\mathcal{C} \ge 255$.
For each hyperplane $H$, there is a point $P' \not \le H$ with $\#\I{\mathcal{C}}{P'} \ge 14$.
Moreover, if $\#\I{\mathcal{C}}{P} \le 16$ for all points $P$, then for each hyperplane $H$ there is a point $P''\not \le H$ with $\#\I{\mathcal{C}}{P''} \ge 15$.
\end{lemma}
\begin{proof}
Let $H$ be a hyperplane and $\mathcal{P} = \gaussmset{\F_2^8}{1}$ be the set of points.
Double counting of the set $\{(P,U) \in \mathcal{P} \times \mathcal{C} \mid P \leq U\}$ gives
\[
\!\!\!\!\!\sum_{P \in \I{\mathcal{P}}{H}}\!\!\!\!\! \#\I{\mathcal{C}}{P} + \!\!\!\!\!\sum_{P \not\in \I{\mathcal{P}}{H}} \!\!\!\!\!\#\I{\mathcal{C}}{P} = \#\mathcal{C} \cdot \gaussmnum{4}{1}{2} \geq 255\cdot 15\text{.}
\]
By Corollary~\ref{cor_one_incidence}, $\#\I{\mathcal{C}}{P} \le 17$ for all points $P$.
Assuming $\#\I{\mathcal{C}}{P} \le 13$ for all points $P\not\le H$, the left hand side is $\le 127 \cdot 17 + 128 \cdot 13 = 255\cdot 15 - 2$, which is a contradiction.
If $\#\I{\mathcal{C}}{P} \leq 16$ for all points $P$, the assumption $\#\I{\mathcal{C}}{P} \leq 14$ for all $P\not\le H$ leads to a left hand side $\le 127 \cdot 16 + 128 \cdot 14 = 255 \cdot 15 - 1$, which is again a contradiction.

\end{proof}

Furthermore we need the following lemma to split a difficult problem into multiple small problems.

\begin{lemma}\label{lem:split}
  Let $X$ be a finite set and $f\colon 2^X\to\{0,1\}$ be a function. A bijection $\pi\colon X\to X$ is called an automorphism (with respect to $f$) if 
  $f(S)=f(\pi(S))$ for all $S\subseteq X$. Let $\Gamma$ be a group of automorphisms, $T=\{t_1, \ldots, t_m\}$ be a transversal of $\Gamma$ 
  acting on $X$, where the corresponding orbit sizes are decreasing, and $\tau\colon X\to \{1,\dots, m\}$ such that $x\in X$ is in the same orbit 
  as $t_{\tau(x)}$. If $\tilde{S}\subseteq X$ 
  and $i=\min\{\tau(x) \mid x\in \tilde{S}\}$, then there exists an automorphism $\gamma\in\Gamma$ with $\{t_i\}\subseteq \gamma(\tilde{S})$, 
  $f(\tilde{S})=f(\gamma(\tilde{S}))$, and $\min\{\tau(x) \mid x\in \gamma(\tilde{S})\}=i$.
\end{lemma}
\begin{proof}
  Choose $x\in X$ with $\tau(x)=i$ and $\gamma\in\Gamma$ with $\gamma(x)=t_i$. Note that $\tau(\gamma'(x'))=\tau(x')$ for all 
  $\gamma'\in\Gamma$ and all $x'\in X$.
  
\end{proof}

This lemma will be applied in Section~\ref{sec_main_thm} to exploit the symmetry for the computation of representatives of cliques of maximal size as well as for the solving of a binary linear program, cf. Section~\ref{sec_ILP}.
If $G=(V,E)$ is a graph with nontrivial automorphism group $\Aut(G)$, we use Lemma~\ref{lem:split} with $X=V$, $f$ defined by $f(S)=1$ iff $S$ is a clique, and $\Gamma \le \Aut(G)$.
Let $T$ be a transversal for the action of $\Gamma$ on $V$.
Then any nonempty clique of size $c$ is in the same orbit as the clique $\{t\} \dot\cup S'$ with $t \in T$ and $\#S' = c-1$.
This argument can also be applied recursively.

\section{General outline of the proof of Theorem~\ref{main_thm}} \label{sec_general}
In the first phase we try to extend the $715+14445$ hyperplane configurations from Theorems~\ref{theo:7_17_6_3_2} and~\ref{theo:7_16_6_3_2}
to $(8,N,6;4)_2$ CDCs with $N \ge 257$. This is accomplished by using the linear programming relaxation of the integer linear programming model from Lemma~\ref{lem:phase_1}.
It turns out that such an extension is not possible for all but $38$ of those hyperplane configurations.

For the remaining $38$ hyperplane configurations the integer linear programming model for the extension to an $(8,N,6;4)_2$ CDC with $N \ge 257$ is
used. This test fails for all but seven of the $38$ cases.

In the second phase we try to enlarge the remaining hyperplane configurations to larger substructures.
Overall, we get $73\,234$ possible $31$-point-hyperplane configurations.

In the third phase it is again tested if these configurations can be extended to $(8,N,6;4)_2$ CDCs with $N \ge 257$.
For this, the linear programming relaxation of the integer linear programming model from Lemma~\ref{lem:phase_2} is used.
All but three hyperplane configurations with $195 + 98 + 240$ $31$-point-hyperplane configurations fail this test.

Finally, the integer linear programming model shows that from the remaining $195 + 98 + 240$ cases exactly two
give $(8,257,6;4)_2$ CDCs. All other configurations lead to smaller codes.

\section{Substructures of $(8,N,6;4)_2$ CDCs for $N\ge 257$}
\label{sec_substructures}

Let $\mathcal{C}$ be an $(8,N,6;4)_2$ CDC with $N\ge 257$. From Corollary~\ref{cor_one_incidence} we 
conclude $\#\I{\mathcal{C}}{H}\le 17$ for any hyperplane $H$. If $\#\I{\mathcal{C}}{H}\le 15$ for each 
hyperplane $H$, then $\# \mathcal{C}\le \gaussmnum{8}{1}{2}\cdot 15 / \gaussmnum{4}{1}{2}=255<257$, since every solid is 
contained in $\gaussmnum{8-3}{7-4}{2}=\gaussmnum{4}{1}{2}$ hyperplanes. So, there exists at least one hyperplane $H$ 
with $\#\I{\mathcal{C}}{H}\in\{16,17\}$. Since $\PGammaL(\F^{8}_{2}) = \operatorname{GL}(\F^{8}_{2})$ acts transitively on the set of hyperplanes, we can assume 
$\#\I{\mathcal{C}}{\widetilde{H}}\in\{16,17\}$. Then $\left(\iota^{-1}\left(\I{\mathcal{C}}{\widetilde{H}}\right)\right)^\perp$, i.e., the corresponding dual in $\widetilde{H}$,  
is a set of pairwise disjoint planes in $\widetilde{H}$, i.e., a $(7,N',6;3)_2$ CDC with $N'\in\{16,17\}$, which have already been classified:

\begin{theorem}(\cite[Theorem 1]{honold2016classification})\label{theo:7_17_6_3_2}
\label{thm_class_1}
$A_2(7,6;3)=17$ and there are $715$ isomorphism types of $(7,17,6;3)_2$ CDCs. Their automorphism groups have orders: 
$1^{551}\allowbreak{}2^{70}\allowbreak{}3^{27}\allowbreak{}4^{19}\allowbreak{}6^{6}\allowbreak{}7^{1}\allowbreak{}8^{8}\allowbreak{}12^{2}\allowbreak{}16^{7}\allowbreak{}24^{6}\allowbreak{}32^{5}\allowbreak{}42^{1}\allowbreak{}48^{5}\allowbreak{}64^{2}\allowbreak{}96^{1}\allowbreak{}112^{1}\allowbreak{}128^{1}\allowbreak{}192^{1}\allowbreak{}2688^{1}$.
\end{theorem}

\begin{theorem}(\cite[Theorem 2]{honold2016classification})\label{theo:7_16_6_3_2}
\label{thm_class_2}
There are $14445$ isomorphism types of $(7,16,6;3)_2$ CDCs. Their automorphism groups have orders: 
$1^{13587}\allowbreak2^{511}\allowbreak3^{143}\allowbreak4^{107}\allowbreak6^{20}\allowbreak7^{4}\allowbreak8^{19}\allowbreak9^{3}\allowbreak12^{24}\allowbreak16^{1}\allowbreak18^{1}\allowbreak20^{1}\allowbreak21^{1}\allowbreak24^{9}\allowbreak36^{1}\allowbreak42^{1}\allowbreak48^{3}\allowbreak64^{1}\allowbreak96^{1}\allowbreak112^{1}\allowbreak168^{2}\allowbreak288^{1}\allowbreak384^{1}\allowbreak960^{1}\allowbreak2688^{1}$.
\end{theorem}

We call those configurations \emph{hyperplane configurations} and denote a transversal of the isomorphism classes of sets of planes in 
Theorems~\ref{theo:7_17_6_3_2} and~\ref{theo:7_16_6_3_2} by $\mathcal{A}_{17}$ and  $\mathcal{A}_{16}$, respectively. So, 
$\left(\iota^{-1}\left(\I{\mathcal{C}}{\widetilde{H}}\right)\right)^\perp$ is isomorphic to exactly one set in $\mathcal{A}_{16} \cup \mathcal{A}_{17}$.  
Computing the LP relaxation of a suitable integer linear programming formulation, see the next section, one can check easily that all but $38$ of the 
$715+14445$ hyperplane configurations can not be extended to $(8,257,6;4)_2$ CDCs. These $38$ remaining elements are listed in Table~\ref{tab:ausgeschrieben} and their LP values are stated Table~\ref{tab:details}. By $F_i$ we denote the corresponding sets of solids in $\F_2^8$ for $1 \le i \le 38$.

Next we want to enlarge some of the possible hyperplane configurations to larger substructures, more precisely those with indices $1\le i\le 7$ in 
Table~\ref{tab:ausgeschrieben}.
Therefore we distinguish both possibilities for $\#\I{\mathcal{C}}{\widetilde{H}}$. If it is $17$, then Lemma~\ref{lem_minilemma} guarantees a point $P \not \le \widetilde{H}$ such that $\#\I{\mathcal{C}}{\widetilde{H}}+\#\I{\mathcal{C}}{P} \ge 17+14=31$. If $\#\I{\mathcal{C}}{\widetilde{H}}=16$ then we can assume w.l.o.g. that $\#\I{\mathcal{C}}{P}\le 16$ for all points $P$, since otherwise we can apply the orthogonality and have the first case. Then Lemma~\ref{lem_minilemma} guarantees a point $P \not \le \widetilde{H}$ such that $\#\I{\mathcal{C}}{\widetilde{H}}+\#\I{\mathcal{C}}{P} \ge 16+15=31$.
Since the stabilizer of $\widetilde{H}$ in $\operatorname{GL}(\F^{8}_{2})$ acts transitively\footnote{
Since
$\operatorname{Stab}_{\operatorname{GL}\left(\F^{8}_{2}\right)}\left(\widetilde{H}\right)
=
\left\{
\left(\begin{smallmatrix}A & 0 \\ b & 1\end{smallmatrix}\right)
\in \operatorname{GL}\left(\F^{8}_{2}\right)
\,\left|\,
A \in \operatorname{GL}\left(\F^{7}_{2}\right) \text{ and } b\in \F_2^{7}
\right.\right\}
$,
any point that is not incident with $\widetilde{H}$, i.e., $\langle(p\mid 1)\rangle$ with $p \in \F_2^{7}$, can be mapped via $\left(\begin{smallmatrix}I_7 & 0 \\ p & 1\end{smallmatrix}\right)^{-1}$ to $\widetilde{P}$.
}
on the set of points not incident with $\widetilde{H}$, we can assume $\#\I{\mathcal{C}}{\widetilde{P}}+\#\I{\mathcal{C}}{\widetilde{H}}\ge 31$.
We call sets of $a$ solids in $\widetilde{H}$ and $b$ solids containing $\widetilde{P}$ with $16\le a\le 17$, $a+b=31$, and minimum subspace 
distance $d=6$ briefly \emph{$31$-point-hyperplane configurations}. 

Anticipating the results from Section~\ref{sec_main_thm}, we mention that altogether just $242$ non-isomorphic $31$-point-hyperplane configurations can be extended 
to CDCs with cardinality $257$. Moreover, we will verify indirectly that in all those extensions there exists a 
codeword $U$ such that $\mathcal{C}\backslash \{U\}$ is isomorphic to an LMRD code. 

\begin{theorem} (\cite{class_mrd})
  \label{thm_mrd_4_4_3}
  The Gabidulin construction gives the unique isomorphism type of (not
  necessarily linear) $4\times 4$ MRD codes over $\F_2$
  with minimum rank distance $3$.
\end{theorem}

This result has been achieved computationally in the context of the work~\cite{class_mrd}.
However, to make this article as self-contained as possible, we decided to include the idea of the proof.

\begin{proof}
  Let $C$ be a $4\times 4$ MRD code over $\F_2$ of minimum rank
  distance $3$.  Then $\#C = 256$. For each vector $v\in\F_2^4$, there
  are exactly $16$ matrices in $C$ having $v$ as their last row.
  After removing this common row, these $16$ matrices form a binary
  $3\times 4$ MRD code of minimum rank distance $3$.  These MRD codes
  have been classified in~\cite{honold2016classification} into $37$
  isomorphism classes.

  Let $C'$ be one of these codes, extended to size $4\times 4$ by appending the zero vector as the last row to all the matrices in $C'$.
  Up to isomorphism, $C$ is the extension of one of these $37$ codes $C'$ by $256-16 = 240$ matrices.
  In particular, for each $v\in\F_2^4\backslash\{\mathbf{0}\}$, it must be possible to add $16$ matrices of size $4\times 4$ with last row $v$ without violating the rank distance condition.
  For fixed $v$, this question can be formulated as a clique problem:
  We define a graph $G_v$ whose vertex set is given by all $4\times 4$ matrices with last row $v$ having rank distance $\geq 3$ to all matrices in $C'$.
  Two vertices are connected by an edge if the corresponding matrices have rank distance $\geq 3$.
  Now the question is whether all graphs $G_v$, $v\in\F_2^4\backslash\{\mathbf{0}\}$, admit a clique of size $16$.
  Using the software~\cite{cliquer}, we found that out of the $37$ types of codes $C'$, this is possible only for a single type.

  For this remaining type, the full extension problem to a $4\times 4$ MRD code is again formulated as a clique problem.
  The graph is defined in a similar way, but without the restriction on the last row of the matrices in the vertex set.
  This yields a graph with $1920$ vertices.
  The maximum clique problem is solved within seconds for this graph
  \footnote{We noticed that the order of the vertices makes a huge difference for the running time.
  For fast results, matrices with the same last row should be numbered consecutively.}
  The result are $8$~cliques of maximum possible size
  $240$. In other words, there are $8$
  extensions to a rank distance $3$ code of size $16 + 240 = 256$, i.e., an MRD code.
  All $8$ codes turned out to be isomorphic to the Gabidulin code.  
  
\end{proof} 

By the last theorem, in our setting there is only a single type of
LMRD code, which is the lifted Gabidulin code.  It is iso-dual
(isomorphic to its orthogonal code).

\begin{corollary}
  \label{cor_isomorphism_types}
  Let $\mathcal{C}$ be an $(8,257,6;4)_2$ CDC that contains an LMRD
  code $\mathcal{C}'$. Then $\mathcal{C}$ is isomorphic to either
  $\{ \langle (I_4 \mid B) \rangle \mid B \in M \} \cup \{ \langle
  (0_{4\times 4} \mid I_4) \rangle \}$ or
  $\{\langle (I_4 \mid B) \rangle \mid B \in M \} \cup \{ \langle
  (0_{4\times 3} \mid I_4 \mid 0_{4\times 1}) \rangle \}$, where $M$
  is the $4\times 4$ Gabidulin code with minimum rank distance $3$,
  $I_4$ is the $4\times 4$ identity matrix, and $0_{m\times n}$ is the
  $m\times n$ all-zero matrix.
\end{corollary}
\begin{proof}
  From Theorem~\ref{thm_mrd_4_4_3} we conclude that $\mathcal{C}'$ is
  the lifted
  Gabidulin code $M$. The automorphism group $A$ of
    $\mathcal{C}'$ has order $4\cdot 15^2\cdot
    2^8=230\,400$. Identifying $V$ with $\F_{16}\times\F_{16}$ and
    denoting by $\alpha$ a generator of $\F_{16}^\times$, $A$ is
    generated by $(x,y)\mapsto(x^2,y^2)$, $(x,y)\mapsto(\alpha x,y)$,
    $(x,y)\mapsto(x,\alpha y)$, and the ``translations''
    $(x,y)\mapsto (x,a_0x+a_1x^2+y)$ with $a_0x+a_1x^2\in M$. From
    this it is readily seen that $A$ partitions the $451$ solids
    intersecting each codeword of $\mathcal{C}'$ in at most a point
    (these are precisely the solids intersecting the special solid $S$
    of $\mathcal{C}'$ in at least a plane) into
    two orbits: An orbit of size $1$ containing $S$, which is fixed by
    $A$, and an orbit of size $450$ containing the solids that meet
    $S$ in a plane. This accounts for the two indicated isomorphism classes of
    $\mathcal{C}$. 
\end{proof}

\section{Integer linear programming models}
\label{sec_ILP}

It is well known that the determination of $A_q(v,d;k)$ can be formulated as an integer linear programming 
problem with binary variables (BLP). If all constraints of the form $x\in\{0,1\}$ are replaced by $x\in\mathbb{R}_{\ge 0}$, we 
speak of the corresponding linear programming relaxation (LP). Suppose that we already know that a CDC $\mathcal{C}$ contains the 
solids from $F \subseteq \gaussmset{\mathbb{F}_2^8}{4}$ and that each point and hyperplane is incident with at most $f$ codewords. Then 
we can state the following upper bounds on $\# \mathcal{C}$:  

\begin{lemma} \label{lem:phase_1}
Let $F \subseteq \gaussmset{\mathbb{F}_2^8}{4}$ and $f\in\mathbb{N}$.
Then any $(8,\#\mathcal{C},6;4)_2$ CDC $\mathcal{C}$ with 
$F \subseteq \mathcal{C}$ and such that each point and hyperplane is incident with at most $f$ codewords has $\#\mathcal{C} \le z_8^{\operatorname{BLP}}(F,f) \le z_8^{\operatorname{LP}}(F,f)$, where $\operatorname{Var}_8 = \gaussmset{\mathbb{F}_2^8}{4}$, $z_8^{\operatorname{LP}}$ is the LP relaxation of $z_8^{\operatorname{BLP}}$, and
\begin{align*}
z_8^{\operatorname{BLP}}(F,f) := \max
\sum_{U \in \operatorname{Var}_8} &x_U \\
\text{subject to}
\sum_{U \in \I{\operatorname{Var}_8}{W}} &x_U \le f			&&\forall W \in \gaussmset{\mathbb{F}_2^8}{w} 	&&\forall w \in \{1,7\} \\
\sum_{U \in \I{\operatorname{Var}_8}{W}} &x_U \le 1			&&\forall W \in \gaussmset{\mathbb{F}_2^8}{w} 	&&\forall w \in \{2,6\} \\
&x_U = 1 													&&\forall U \in F \\
&x_U \in \{0,1\} 											&&\forall U \in \operatorname{Var}_8.
\end{align*}
\end{lemma}
\begin{proof}
Interpreting $(x_U)_{U \in \operatorname{Var}_8}$ as characteristic vector of $\mathcal{C}$,
the objective function 
equals $\#\mathcal{C}$.
The first two sets of constraints are feasible by Lemma~\ref{lem:deg_upper_bound} and the choice of $f$.
The third set of constraints is feasible since $F \subseteq \mathcal{C}$.

\end{proof}

If $\# F$ is rather small, then the computation of 
$z_8^{\operatorname{BLP}}(F,f)$ takes too much time, so that we also consider a linear programming 
formulation for $\# \{U\cap \widetilde{H} \mid U\in \mathcal{C}\}$, i.e., we consider the image of $\mathcal{C}$ in $\widetilde{H}$.  

\begin{lemma} \label{lem:phase_2}
For $F \subseteq \gaussmset{\mathbb{F}_2^7}{4}$ let $\operatorname{Var}_7(F):=\left.\left\{ U \in \gaussmset{\mathbb{F}_2^7}{3} \,\right|\, \dim(U \cap S) \le 1 \,\forall S \in F \right\}$ and 
$\omega(F,W) = \max\{ \#\Omega \mid \Omega \subseteq \I{\operatorname{Var}_7(F)}{W} \land \dim(U_1 \cap U_2) \le 1 \,\forall U_1 \ne U_2 \in \Omega \}$. 
If $\#F \in \{16,17\}$, then any $(8,\#\mathcal{C},6;4)_2$ CDC $\mathcal{C}$ with $\#\mathcal{C} \ge 255$ and $\iota(F) \subseteq \mathcal{C}$ and such that each point 
and hyperplane is incident with at most $\#F$ codewords satisfies $\#\mathcal{C} \le z_7^{\operatorname{BLP}}(F)$, where
\begin{align*}
z_7^{\operatorname{BLP}}(F) := \max \!\!\!\!\!\!\!
\sum_{U \in \operatorname{Var}_7(F)} &x_U + \#F \\
\text{subject to}
\sum_{U \in \I{\operatorname{Var}_7(F)}{W}} &x_U \le \#F-\#\I{F}{W}						&&\forall W \in \gaussmset{\mathbb{F}_2^7}{1} \\
\sum_{U \in \I{\operatorname{Var}_7(F)}{W}} &x_U \le 1									&&\forall W \in \gaussmset{\mathbb{F}_2^7}{2} \setminus (\cup_{S \in F} \gaussmset{S}{2}) \\
\sum_{U \in \I{\operatorname{Var}_7(F)}{W}} &x_U \le 1									&&\forall W \in \gaussmset{\mathbb{F}_2^7}{4} \setminus F \\  
\sum_{U \in \I{\operatorname{Var}_7(F)}{W}} &x_U \le \min\{\omega(F,W),7\}				&&\forall W \in \gaussmset{\mathbb{F}_2^7}{5} : S \not \le W \,\forall S \in F  \\
\sum_{U \in \I{\operatorname{Var}_7(F)}{W}} &x_U \le 2(\#F-\#\I{F}{W})					&&\forall W \in \gaussmset{\mathbb{F}_2^7}{6} \\
\sum_{U \in \operatorname{Var}_7(F)} &x_U + \#F \ge 255 \\
&x_U \in \{0,1\} 																		&&\forall U \in \operatorname{Var}_7(F) \\
\end{align*}
\end{lemma}
\begin{proof}
Interpreting $(x_U)_{U \in \operatorname{Var}_7(F)}$ as characteristic vector of $\{ U \cap \widetilde{H} \mid U \in \mathcal{C} \land U \not \le \widetilde{H} \}$, one can check the correctness of the objective function and the last two lines.
Since two solids in $\mathcal{C}$ intersect in at most a point, any two elements in $\{ U \cap \widetilde{H} \mid U \in \mathcal{C} \}$ also intersect in at most a point, which gives the constraints with $\dim(W) \in \{2,4\}$.

Any $5$-space $W$ contains at most $\omega(F,W)$ planes by choice of $\omega$, also $\iota(W)$ is incident with $\gaussmnum{8-5}{6-5}{2}=7$ $6$-spaces, which in turn contain at most one codeword of $\mathcal{C}$.
If $W$ contains a solid of $F$, then any plane in $W$ meets this solid in at least a line.
This gives the constraints with $\dim(W) = 5$.

For any point $W$ its embedding $\iota(W)$ is incident with at most $\#F$ codewords of $\mathcal{C}$ giving the constraints with $\dim(W) = 1$.

For any $6$-subspace $W$ its embedded $\iota(W)$ is contained in $\gaussmnum{8-6}{7-6}{2}=3$ hyperplanes in $\mathbb{F}_2^8$ of which one of them is $\widetilde{H}$.
Since each hyperplane is incident with at most $\#F$ codewords and $\bar{H}$ is incident with exactly $\#F$ codewords, i.e., $\iota(F)$, the other two hyperplanes are each incident with either $\#F$ codewords if $W$ contains no element of $F$ or $\#F-1$ codewords if $W$ contains one element of $F$.
Obviously two solids in a $6$-space intersect in at least a line and hence $W$ contains at most one element of $F$.
This gives the constraints with $\dim(W) = 6$.

The last inequality allows the BLP solver to cut the branch \& bound tree early since we are only interested in solutions of cardinality at least $255$.

\end{proof}

\section{Proof of the main theorem}
\label{sec_main_thm}

The algorithmic proof of Theorem~\ref{main_thm} is split into several phases that are described in detail in the following 
subsections; Subsection~6.i corresponds to Phase~i. The
(integer) linear programming problems are solved with
\texttt{CPLEX}~\cite{citeulike:8436868}.

Let $\mathcal{C}$ be an $(8,\#\mathcal{C},6;4)_2$ CDC with $\#\mathcal{C}\ge 257$. As argued at the beginning of Section~\ref{sec_substructures}, 
$\mathcal{C}$ has to contain one of the $715+14445$ hyperplane configurations from $\mathcal{A}_{17}\cup\mathcal{A}_{16}$. 
This list is reduced in Phase~1, see Section~\ref{subsec_phase_1}, and then extended to $31$-point-hyperplane configurations 
in Phase~2, see Section~\ref{subsec_phase_2}. The
resulting list is reduced in Phase~3, see
Section~\ref{subsec_phase_3}. Then 
we deduce that $\mathcal{C}$ must be an LMRD code extended by a single codeword, see Section~\ref{subsec_phase_4}. The classification of 
such structures at the end of Section~\ref{sec_substructures} concludes the proof. Let us mention that the termination of Phase~1 proves 
$A_2(8,6;4)\le 271$ and the termination of Phase~3 proves $A_2(8,6;4)=257$. The required computation times for the four phases are
$42\,087$, $2\,214$, $1\,804$, and $2\,168$~hours, respectively, i.e., $48\,273$~hours in total.

Besides the internal parallelization performed by the ILP solvers, we employed parallelization only by setting up independent subproblems.
We used the cluster of the University of Bayreuth\footnote{\url{http://www.hpc.uni-bayreuth.de}} for solving the subproblems and other computers for the management and generation of the subproblems.

\subsection{Excluding hyperplane configurations}
\label{subsec_phase_1}
For all $A \in \mathcal{A}_{16} \cup \mathcal{A}_{17}$ we computed $z_8^{\operatorname{LP}}(\iota(A^\perp),\#A)$ and found that all but 
$33$ elements in $\mathcal{A}_{16}$ ($37\,251$ hours) and $5$ elements in $\mathcal{A}_{17}$ ($1021$ hours) have an optimal value smaller than $256.9$, i.e., we have 
implemented  a safety threshold of $\varepsilon=0.1$.
These $38$ elements are listed in Table~\ref{tab:ausgeschrieben} and their LP values are stated in Table~\ref{tab:details}.

For indices $1\le i\le 38$ we computed $z_7^{\operatorname{BLP}}(\iota(F_i))$ and obtained 
$6$ elements in $\mathcal{A}_{16}$ and $2$ elements in $\mathcal{A}_{17}$ that may allow $z_7^{\operatorname{BLP}}(\iota(F_i))\ge 256.9$, 
cf. Table~\ref{tab:details} for details. This computation was aborted after $100$~hours of wall time for each of these $38$ subproblems.

$\operatorname{Var}_7(\iota(F_8))$ has exactly $948$ planes which form $56$ orbits ($4^3 8^{13} 16^{28} 32^{12} $) under the action of 
the automorphism group of order $32$.
We apply Lemma~\ref{lem:split} to obtain $56$ subproblems.
Less than $15$~hours were needed to verify 
$z_7^{\operatorname{BLP}}\le 256$ in all cases.

\subsection{Extending hyperplane configurations to $31$-point-hyperplane configurations}
\label{subsec_phase_2}
The seven hyperplane configurations, with indices $1\le i\le 7$ remaining after Phase~1, are extended to $31$-point-hyperplane 
configurations.

We define a graph $G_i=(V_i,E_i)$, whose vertex set $V_i$ consists of all solids in $\gaussmset{\mathbb{F}_2^8}{4}$ that contain $\widetilde{P}$ and 
intersect the elements from $F_i$ in at most a point. For $U,W\in V_i$ we set $\{U,W\}\in E_i$ iff $U\cap W=\widetilde{P}$.  
Using \texttt{Cliquer}~\cite{cliquer}, we enumerated all cliques of size $31-\#F_i$ of $G_i$ and computed a transversal $T(F_i)$ of the action of the stabilizer 
of $F_i$. The clique computations for $1\le i\le 7$, $i\neq 5$ took
between $27$ and $589$ hours (see
  Table~\ref{tab:detailsphase2} for details about the running times
  and $\#V_i$; the computation time for the transversal was
  negligible). The transversal is denoted by $T(F_i)$; see Column~6 of Table~\ref{tab:details} 
for the corresponding orbit lengths.

The clique computation for $G_5$ was aborted after $600$~hours and then performed in parallel by applying Lemma~\ref{lem:split} with $X$ as the vertex set of $G_5$, $\Gamma$ the automorphism group of $F_5$, and the function $f$ defined by $f(S)$ equals 
$1$ iff $S$ is a clique in $G_5$. In general, we label the elements of $T$ in decreasing order of the corresponding orbit lengths, since large orbits 
admit small stabilizers and forbid many elements from $X$ in the subsequent subproblems, resulting in few rather asymmetric large subproblems 
and many small subproblems. The $1258$ vertices of $G_5$ are partitioned into $24$ orbits of size $1$ and $617$ orbits of size $2$ by $\Gamma$, 
leaving $641$ graphs where we have to enumerate all cliques of size $31-\# F_5-1=14$. Since some of these graphs still consist of 
\textit{many} vertices, we iteratively apply Lemma~\ref{lem:split} with the identity group as $\Gamma$ for at most two further times: 
After the first round we split the $68$ subproblems, which lead to graphs with at least $700$ vertices. Then, we 
split the $81$ subproblems, which lead to graphs with at least $600$
vertices.
We are left with $104\,029$ graphs, for which we 
have to enumerate all cliques of size $14$, $13$ or $12$. All of these instances were solved in parallel with \texttt{Cliquer} to get a superset 
of the transversal of all cliques of size $15$ of $G_5$. Applying the action of the automorphism group of order~$2$ then allowed us to obtain a 
transversal as well as all cliques, simply as union of the orbits. This took about $750$ hours of CPU time, the smaller problems 
being preprocessed on a single computer and the remaining $55\,420$
larger subproblems being processed in parallel with $16$ cores.

The extension of the configuration with index $5$ took $750$ hours, and the extension of the other indices took $1464$ hours; see Table~\ref{tab:detailsphase2} for details.

\subsection{Excluding $31$-point-hyperplane configurations}
\label{subsec_phase_3}
For the $73\,234$ $31$-point-hyperplane configurations resulting from Section~\ref{subsec_phase_2}, we computed 
$z_8^{\operatorname{LP}}(.)$ in $953$ hours. The maximum value aggregated by the contained hyperplane configuration with index $i$ is stated 
in Column~7 of Table~\ref{tab:details}, see also Table~\ref{tab:detailsphase2}. For the configuration with index $1$ there are $195$, for the configuration with index $3$ there are $98$, and for the configuration with index $7$ there are $240$ 
$31$-point-hyperplane configurations with $z_8^{\operatorname{LP}}\ge 256.9$.

Next we computed $z_8^{\operatorname{BLP}}$ (see Column~8 of Table~\ref{tab:details}) for these remaining  $195+98+240$ cases in $851$ hours (see Table~\ref{tab:detailsphase2}). 

The counts for value exactly $257$ are $2+0+240$. 

\subsection{Structural results for $(8,N,6;4)_2$ CDCs with $N\ge 257$}
\label{subsec_phase_4}
So far we know that the hyperplane configuration of $\mathcal{C}$ in $\tilde{H}$ is either $F_1 \in \mathcal{A}_{16}$ 
or $F_7 \in \mathcal{A}_{17}$ with $2$ and $240$ possible $31$-point-hyperplane configurations, respectively.

For $F_1$ there exists a unique solid $S$ in $\F_2^8$ which is disjoint from the $31$ prescribed solids in both cases. Adding the 
constraint $x_S=0$ to the BLP of Lemma~\ref{lem:phase_1} gives an upper bound of $256$, i.e., $S$ has to be a codeword in $\mathcal{C}$, 
after about $2$~hours of computation time in each of the two cases. The codeword $S$ covers $15$ contained points. Via 
$x_S=1$ and
$
\sum_{P \in \I{\gaussmset{V}{1}}{S}}
\sum_{U \in \I{\operatorname{Var}_8}{P}}
x_U
\ge 16
$  
we can ensure that another codeword of $\mathcal{C}$ meets $S$ in a point. This modification of the BLP of Lemma~\ref{lem:phase_1} gives 
again an upper bound of $256$ after about two hours of computation time in both cases. Thus $\mathcal{C}\backslash\{S\}$ has to be an 
LMRD code.

For $F_7$ there exists a unique solid $S$ in $\F_2^8$ which is disjoint from $30$ of the prescribed solids and meets the other 
prescribed solid $S'$ in a plane, in all $240$ cases. By adding 
$
\sum_{P \in \I{\gaussmset{V}{1}}{S}}
\sum_{U \in \I{\operatorname{Var}_8}{P}}
x_U
\ge 8
$
we can ensure that $S$ meets another codeword from $\mathcal{C}$ in a point. The augmented BLP of Lemma~\ref{lem:phase_1} needs $9$~hours 
computation time and ends with $z_8^{\operatorname{BLP}}\le 256$ for each of the $240$ cases. Thus $\mathcal{C}\backslash\{S'\}$ has to be an 
LMRD code.

\section*{Acknowledgements}
The authors would like to thank the \emph{High Performance Computing
  group} of the University of Bayreuth for providing the excellent
computing cluster and especially Bernhard Winkler for his support.

{\footnotesize

}

\section*{Appendix}

It is well known that any plane in $\mathbb{F}_2^7$ has a unique binary $3 \times 7$ generator matrix in reduced row echelon form and vice versa.
In Table~\ref{tab:ausgeschrieben}, we list the $38$ $(7,16,6;3)_2$ and $(7,17,6;3)_2$ CDCs with $z_8^{\operatorname{LP}}(.)\ge 256.9$.
Each plane is denoted by an integer with at most seven digits, one for each column of the generator matrix in such a way that the three entries in each column are coefficients of a $2$-adic number, i.e., $(c_1, c_2, c_3)^T \leftrightarrow c_1 \cdot 2^0 + c_2 \cdot 2^1 + c_3 \cdot 2^2$.
Leading zeroes are omitted.
For example the number $1024062$ denotes the subspace
$
\left(
\begin{smallmatrix}
1&0&0&0&0&0&0\\
0&0&1&0&0&1&1\\
0&0&0&1&0&1&0\\
\end{smallmatrix}
\right)
$.
Note that since we are encoding matrices in reduced row echelon form, the three pivot columns are the first numbers $1$, $2$, and $4$ appearing in 
this order and no digit is larger than $7$.
Table~\ref{tab:details} lists for these CDCs whether it is in $\mathcal{A}_{16}$ or $\mathcal{A}_{17}$, the size of their automorphism group, the relaxations $z_8^{\operatorname{LP}}(.)$ and $z_7^{\operatorname{BLP}}(.)$, which are applied to the hyperplane configurations, then the orbits of the extension to point-hyperplane configurations of each hyperplane configuration and finally the maximum of $z_8^{\operatorname{LP}}(.)$ with prescribed point-hyperplane configuration grouped by the contained hyperplane configuration and, if needed, the maximum $z_8^{\operatorname{LP}}(.)$, again for prescribed point-hyperplane configuration grouped by the contained hyperplane configuration.
Details for the extension of one of the first seven hyperplane configurations to the corresponding point-hyperplane configurations are shown in Table~\ref{tab:detailsphase2}.

\begin{table}[h]
\caption{Details for the computation of all $31$-point-hyperplane configurations in Phase~2 and Phase~3.}
\label{tab:detailsphase2}
\centering
\begin{tabular}{l|l|lll}
&&\multicolumn{3}{c}{Wall-time in hours for}\\
$i$ & $\#V_i$ & Phase 2 & LP in Phase 3 & BLP in Phase 3 \\
\hline
1 & 1231	& 144	& 51	&	328 \\
2 & 1303	& 589	& 78	&		\\
3 & 1194	& 217	& 21	&	519 \\
4 & 1243	& 278	& 22	&		\\
5 & 1258	& 750	& 419	&		\\
6 & 1251	& 209	& 13	&		\\
7 & 864		& 27	& 349	&	4	\\
\end{tabular}
\end{table}

\begin{sidewaystable}
\caption{Details for the $38$ $(7,16,6;3)_2$ and $(7,17,6;3)_2$ CDCs with $z_8^{\operatorname{LP}}(.)\ge 256.9$.}
\label{tab:details}
\begin{tabular}{ll|l|lllll}
Index	&	Type	&	Aut	&	$z_8^{\operatorname{LP}}(.)$	&	$z_7^{\operatorname{BLP}}(.)$	&	Orbits of Phase~2	&	$\max z_8^{\operatorname{LP}}(\text{\lq\lq 31\rq\rq})$	&	$\max z_8^{\operatorname{BLP}}(\text{\lq\lq 31\rq\rq})$	\\
\hline
1	&	16	&	960	&	272	&	271.1856	&	$ 16^{2},240^{6},480^{47},960^{242} $	&	263.0287799	&	257	\\
2	&	16	&	384	&	266.26086957	&	267.4646	&	$ 96^{6},192^{91},384^{711} $	&	206.04279728	&		\\
3	&	16	&	4	&	270.83786676	&	265.3281	&	$ 1^{13},2^{29},4^{2638} $	&	257.20717665	&	254	\\
4	&	16	&	48	&	271.43451032	&	262.082	&	$ 4^{3},12^{11},24^{59},48^{1104} $	&	200.5850228	&		\\
5	&	16	&	2	&	263.8132689	&	259.8044	&	$ 1^{5},2^{59966} $	&	206.39304042	&		\\
6	&	16	&	20	&	267.53272206	&	259.394	&	$ 5,10^{9},20^{1843} $	&	199.98690666	&		\\
7	&	17	&	64	&	282.96047431	&	259.1063	&	$ 16^{10},32^{145},64^{6293} $	&	259.45364626	&	257	\\
8	&	17	&	32	&	268.0388109	&	257.2408							\\
\hline
9	&	16	&	1	&	263.82742528	&	256.392							\\
10	&	16	&	1	&	263.36961743	&	255.8305							\\
11	&	16	&	1	&	264.25957151	&	$\le$ 254							\\
12	&	16	&	1	&	263.85869815	&	$\le$ 254							\\
13	&	16	&	2	&	263.07052878	&	$\le$ 254							\\
14	&	16	&	12	&	261.91860556	&	$\le$ 254							\\
15	&	16	&	4	&	261.62648174	&	$\le$ 254							\\
16	&	16	&	12	&	261.31512837	&	$\le$ 254							\\
17	&	17	&	4	&	261.11518721	&	$\le$ 254							\\
18	&	16	&	1	&	260.96388752	&	$\le$ 254							\\
19	&	16	&	1	&	260.82432878	&	$\le$ 254							\\
20	&	16	&	2	&	260.65762276	&	$\le$ 254							\\
21	&	16	&	4	&	260.43036283	&	$\le$ 254							\\
22	&	16	&	2	&	260.19475349	&	$\le$ 254							\\
23	&	16	&	1	&	260.08583792	&	$\le$ 254							\\
24	&	16	&	1	&	260.04857193	&	$\le$ 254							\\
25	&	16	&	1	&	259.75041996	&	$\le$ 254							\\
26	&	16	&	2	&	259.55230081	&	$\le$ 254							\\
27	&	16	&	2	&	259.46335297	&	$\le$ 254							\\
28	&	16	&	12	&	259.11945025	&	$\le$ 254							\\
29	&	16	&	1	&	258.89395938	&	$\le$ 254							\\
30	&	17	&	24	&	258.75142045	&	$\le$ 254							\\
31	&	16	&	8	&	258.35689437	&	$\le$ 254							\\
32	&	16	&	1	&	257.81420526	&	$\le$ 254							\\
33	&	16	&	2	&	257.75126819	&	$\le$ 254							\\
34	&	16	&	4	&	257.63965018	&	$\le$ 254							\\
35	&	16	&	1	&	257.57663803	&	$\le$ 254							\\
36	&	16	&	1	&	257.2820438	&	$\le$ 254							\\
37	&	16	&	4	&	257.01931801	&	$\le$ 254							\\
38	&	17	&	128	&	257	&	$\le$ 254							\\
\end{tabular}
\end{sidewaystable}

\begin{sidewaystable}
\thisfloatpagestyle{empty}
\caption{The $38$ $(7,16,6;3)_2$ and $(7,17,6;3)_2$ CDCs with $z_8^{\operatorname{LP}}(.)\ge 256.9$.}
\label{tab:ausgeschrieben}
\hspace*{-3cm}
\begin{tabular}{l|l}
Index & $16$ or $17$ planes in $\F_2^7$ \\
\hline
1 & 1240000,1240124,1241062,1241146,1242463,1242547,1243401,1243525,1244635,1244711,1245657,1245773,1246256,1246372,1247234,1247310 \\
2 & 1240000,1240124,1241062,1241146,1242647,1242763,1243625,1243701,1244234,1244310,1245256,1245372,1246473,1246557,1247411,1247535 \\
3 & 124,1240000,1240124,1241447,1241563,1242631,1242715,1243276,1243352,1244230,1244314,1245753,1246401,1246525,1247046,1247162 \\
4 & 1240000,1240524,1241042,1241566,1242237,1242403,1243165,1243751,1244270,1244354,1245632,1245716,1246127,1246313,1247441,1247675 \\
5 & 124,1240124,1241046,1241162,1242637,1242713,1243671,1243755,1244230,1244314,1245276,1245352,1246407,1246523,1247441,1247565 \\
6 & 1240000,1240124,1241370,1241757,1242605,1242721,1243276,1243451,1244017,1244133,1245263,1245345,1246534,1246612,1247446,1247562 \\
7 & 124,124000,124124,1024062,1024146,1214452,1214746,1224403,1224727,1241572,1241633,1242557,1242615,1245461,1245724,1246476,1246730 \\
8 & 124,124000,124124,1024062,1024146,1214546,1214652,1224503,1224627,1241471,1241730,1242416,1242754,1245527,1245662,1246575,1246633 \\
9 & 124,1240000,1240124,1241157,1242634,1242756,1243673,1243710,1244211,1244335,1245262,1245347,1246463,1246501,1247425,1247546 \\
10  & 124,1240000,1240124,1241072,1241157,1242634,1242756,1243673,1243710,1244211,1244335,1245347,1246463,1246501,1247425,1247546 \\
11  & 124,1240000,1241072,1241157,1242634,1242756,1243673,1243710,1244211,1244335,1245262,1245347,1246463,1246501,1247425,1247546 \\
12  & 124,1240000,1240124,1241072,1241157,1242634,1242756,1243673,1243710,1244211,1245262,1245347,1246463,1246501,1247425,1247546 \\
13  & 124,1240000,1240124,1241241,1241630,1242415,1242561,1243166,1244023,1244452,1245613,1245737,1246354,1246775,1247206,1247372 \\
14  & 124,1240000,1240124,1241241,1241630,1242415,1242561,1243166,1243547,1244023,1244452,1245737,1246354,1246775,1247206,1247372 \\
15  & 124,1240000,1241437,1241513,1242661,1242745,1243252,1243376,1244230,1244314,1245647,1245763,1246051,1246175,1247422,1247506 \\
16  & 124,1240000,1241241,1241630,1242415,1242561,1243166,1243547,1244023,1244452,1245613,1245737,1246354,1246775,1247206,1247372 \\
17  & 124,124000,124124,1024466,1024553,1204267,1204342,1234506,1234713,1240570,1240721,1243437,1243565,1245042,1245126,1246453,1246634 \\
18  & 124,1240000,1241664,1241740,1242427,1242503,1243165,1243243,1244076,1244757,1245516,1245632,1246372,1246451,1247235,1247311 \\
19  & 124,1240000,1240124,1241367,1241446,1242521,1243243,1243562,1244076,1244757,1245311,1245734,1246150,1246673,1247235,1247412 \\
20  & 124,1240000,1240124,1241367,1241446,1242521,1242605,1243243,1243562,1244757,1245311,1245734,1246150,1246673,1247235,1247412 \\
21  & 124,1240000,1240124,1241367,1241446,1242521,1242605,1243243,1243562,1244076,1244757,1245311,1245734,1246150,1247235,1247412 \\
22  & 124,1240000,1240124,1241664,1241740,1242427,1242503,1243165,1244076,1244757,1245516,1245632,1246372,1246451,1247235,1247311 \\
23  & 124,1240000,1240124,1241664,1241740,1242427,1242503,1243165,1243243,1244076,1244757,1245516,1245632,1246372,1246451,1247311 \\
24  & 124,1240000,1240124,1241367,1241446,1242521,1242605,1243243,1244076,1244757,1245311,1245734,1246150,1246673,1247235,1247412 \\
25  & 124,1240000,1240124,1241664,1241740,1242427,1242503,1243165,1243243,1244076,1244757,1245516,1245632,1246372,1247235,1247311 \\
26  & 124,1240000,1240124,1241664,1242427,1242503,1243165,1243243,1244076,1244757,1245516,1245632,1246372,1246451,1247235,1247311 \\
27  & 124,1240000,1240124,1241740,1242427,1242503,1243165,1243243,1244076,1244757,1245516,1245632,1246372,1246451,1247235,1247311 \\
28  & 124,1240000,1240124,1241437,1241513,1242661,1242745,1243376,1244230,1244314,1245647,1245763,1246051,1246175,1247422,1247506 \\
29  & 124,1240124,1241664,1241740,1242427,1242503,1243165,1243243,1244076,1244757,1245516,1245632,1246372,1246451,1247235,1247311 \\
30  & 124,124000,124124,1024341,1024630,1204526,1204653,1234367,1234644,1240046,1240135,1243474,1243726,1245237,1245664,1246512,1246605 \\
31  & 124,1240000,1240124,1241057,1241173,1242655,1242771,1243602,1243726,1244230,1244314,1245267,1245343,1246465,1246541,1247516 \\
32  & 124,1240000,1240124,1241664,1241740,1242427,1242503,1243165,1243243,1244076,1245516,1245632,1246372,1246451,1247235,1247311 \\
33  & 124,1240000,1240124,1241664,1241740,1242427,1242503,1243243,1244076,1244757,1245516,1245632,1246372,1246451,1247235,1247311 \\
34  & 124,1240000,1240124,1241367,1241446,1242521,1242605,1243243,1243562,1244076,1244757,1245311,1245734,1246673,1247235,1247412 \\
35  & 124,1240000,1240124,1241367,1241446,1242521,1242605,1243243,1243562,1244076,1244757,1245311,1245734,1246150,1246673,1247235 \\
36  & 124,1240000,1240124,1241664,1241740,1242427,1242503,1243165,1243243,1244076,1244757,1245632,1246372,1246451,1247235,1247311 \\
37  & 10024,1202436,1211471,1221433,1232464,1240776,1243450,1243712,1244143,1244522,1245307,1245660,1246021,1246615,1247267,1247546 \\
38  & 124,124000,124124,1024062,1024146,1214466,1214772,1224437,1224713,1241561,1241620,1242574,1242636,1245407,1245742,1246423,1246765 \\
\end{tabular}
\end{sidewaystable}

\end{document}